\documentclass[preprint,3p]{elsarticle}
\usepackage{amssymb}
\usepackage{amsmath,amsthm,amscd,amssymb,enumitem,latexsym,upref,xfrac}
\usepackage{graphicx}
\usepackage{color}
\usepackage{epstopdf}
\usepackage{pdfsync}
\usepackage{amssymb,amsfonts,amsthm,amssymb}
\usepackage{amsmath}
\usepackage{mathrsfs}
\usepackage{mathtools}
\usepackage{calc}
\usepackage{verbatim}
\usepackage{stmaryrd}
\usepackage{xcolor}
\usepackage{transparent} 
\newcommand*{\mailto}[1]{\href{mailto:#1}{\nolinkurl{#1}}}

\newcommand{\R}{{\mathbb{R}}}
\newcommand{\one}{\chi}

\newtheorem{theorem}{Theorem}
\newtheorem{lemma}{Lemma}
\newtheorem{proposition}{Proposition}
\newtheorem{corollary}{Corollary}
\newtheorem{remark}{Remark}

\def\be#1{\begin{equation}\label{#1}}
\def\ee{\end{equation}}

\def\bea{\begin{eqnarray}}
\def\eea{\end{eqnarray}}
\def\bth#1{\begin{theorem}\label{#1}}
\def\eth{\end{theorem}}
\def\brem#1{\begin{remark}\label{#1}}
\def\erem{\end{remark}}
\def\blem#1{\begin{lemma}\label{#1}}
\def\elem{\end{lemma}}
\def\beas{\begin{eqnarray*}}
\def\eeas{\end{eqnarray*}}
\def\barr{\begin{array}}
\def\earr{\end{array}}
\def\bdm{\begin{displaymath}}
\def\edm{\end{displaymath}}
\def\bcor#1{\begin{corollary}\label{#1}}
\def\ecor{\end{corollary}}
\def\cW{\mathcal{W}}

\def\div{\mathrm{div}}

\numberwithin{equation}{section}
\allowdisplaybreaks

\begin{document}

\begin{frontmatter}

\title{On higher regularity for the Westervelt equation with strong nonlinear damping}

\author{Vanja Nikoli\' c \corref{cor1}}

\address{Insitut f\"ur Mathematik, Alpen-Adria-Universit\"at Klagenfurt\\
Universit\"atsstra{\ss}e 65-57, 9020 Klagenfurt am W\"orthersee,  Austria}
\ead{vanja.nikolic@aau.at}

\author{Barbara Kaltenbacher}
\address{Insitut f\"ur Mathematik, Alpen-Adria-Universit\"at Klagenfurt\\
Universit\"atsstra{\ss}e 65-57, 9020 Klagenfurt am W\"orthersee,  Austria}
\ead{barbara.kaltenbacher@aau.at}
\cortext[cor1]{Corresponding Author: Vanja Nikoli\' c, Universit\"atsstra{\ss}e 65-57, 9020 Klagenfurt am W\"orthersee,  Austria; vanja.nikolic@aau.at; +43 463 2700 993133}
\begin{abstract}
\indent We show higher interior regularity for the Westervelt equation with strong nonlinear damping term of the $q$-Laplace type. Secondly, we investigate an interface coupling problem for these models, which arise, e.g., in the context of medical applications of high intensity focused ultrasound in the treatment of kidney stones. We show that the solution to the coupled problem exhibits piecewise $H^2$ regularity in space, provided that the gradient of the acoustic pressure is essentially bounded in space and time on the whole domain. This result is of importance in numerical approximations of the present problem, as well as in gradient based algorithms for finding the optimal shape of the focusing acoustic lens in lithotripsy.
\end{abstract}


\end{frontmatter}

\section{Introduction}
High intensity focused ultrasound has numerous applications starting from the treatment of kidney and bladder stones, via thermo therapy and ultrasound cleaning to sonochemistry. Due to the nonlinear effects observed in the propagation of ultrasound in these cases, such as the appearance of sawtooth solutions,  models of nonlinear acoustics and their rigourous mathematical treatment have become of great interest in recent years. \\
\indent One of the most popular models for the nonlinear propagation of ultrasound is the Westervelt equation
\begin{align} \label{westervelt}
(1-2ku)\ddot{u} -c^2 \Delta u - b\Delta \dot{u} = 2k(\dot{u})^2,
\end{align}
expressed here in terms of the acoustic pressure $u$, where a dot denotes time differentiation, $b$ the diffusivity and $c$ the speed of sound, $k = \beta_a / \lambda$, $\lambda=\varrho c^2$ is the bulk modulus, $\varrho$ is the mass density, $\beta_a = 1 + B/(2A)$, and $B/A$ represents the parameter of nonlinearity. A detailed derivation of \eqref{westervelt} can be found in \cite{HamiltonBlackstock}, \cite{manfred} and \cite{Westervelt}.\\
\indent  Westervelt's equation is a quasilinear wave equation which can degenerate due to the factor $1-2ku$. This means that any analysis of this equation has to include bounding away from zero this term, i.e. finding an essential bound for $u$.  That has been so far  achieved by means of employing the Sobolev embedding $H^2(\Omega) \hookrightarrow L^{\infty}(\Omega)$ (cf. \cite{KL08}, \cite{KLV10}), which implies that the solution of the Westervelt equation has to exhibit $H^2$ regularity in space. However, achieving $H^2$-regularity is too high of a demand in the case of coupling acoustic regions with different material parameters.\\
\indent To remedy this issue, Westervelt's equation is considered with an added nonlinear damping term
\begin{align} \label{westervelt_damp}
(1-2ku)\ddot{u}-c^2\Delta u -\div(b((1-\delta) +\delta|\nabla \dot{u}|^{q-1})\nabla \dot{u})
=2k(\dot{u})^2,
\end{align} 
cf., \cite{BKR13}, \cite{VN},
where $\delta \in (0,1)$, $q \geq 1$, $q >d-1$, and $d \in \{1,2,3\}$ is the dimension of the spatial domain $\Omega\subseteq\R^d$ on which \eqref{westervelt_damp} is considered. Since degeneracy is avoided with the help of the embedding $W^{1,q+1}(\Omega) \hookrightarrow L^\infty(\Omega)$, using this model allows to show existence of weak solutions with $W^{1,q+1}$ regularity in space, and in turn well-posedness of the acoustic-acoustic coupling problem.\\
\indent Considerations of the acoustic-acoustic coupling are motivated by lithotripsy where a silicone acoustic lens focuses the ultrasound traveling through a nonlinearly acoustic fluid to a kidney stone (see fig. 1). The interface coupling is  modeled by the presence of spatially varying coefficients in the weak form of the equation \eqref{westervelt_damp} (see \cite{BGT97} for the linear and \cite{BKR13} and \cite{VN} for the nonlinear case) as follows:
\begin{equation}\label{ModWest_coupled}
\begin{cases} 
\text{Find} \ u  \ \text{such that} \vspace{1.5mm}\\
\int_0^T \int_{\Omega} \{\frac{1}{\lambda(x)}(1-2k(x)u)\ddot{u} \phi+\frac{1}{\varrho(x)}\nabla u \cdot \nabla \phi + b(x)(1-\delta(x))\nabla \dot{u} \cdot \nabla \phi \vspace{1.5mm}\\
\quad \quad \quad +b (x)\delta(x)|\nabla \dot{u}|^{q-1} \nabla \dot{u} \cdot \nabla \phi-\frac{2k(x)}{\lambda(x)}(\dot{u})^2 \phi\} \, dx \, ds =0 \vspace{1.5mm}\\
\text{holds for all test functions} \ \phi \in \tilde{X}=L^2(0,T;W_0^{1,q+1}(\Omega)),
\end{cases}
\end{equation}
with $(u,\dot{u})\vert_{t=0}=(u_0,u_1)$. In this model $b$ stands for the quotient between the diffusivity and the bulk modulus, while the other coefficients maintain their meaning. The coefficients are allowed to jump only over the interface, i.e. the boundary of the lens. For notational brevity, we emphasized the space dependence of coefficients in \eqref{ModWest_coupled}, while omitting space and time dependence of $u$ in the notation. \vspace{2.5mm} \\
\begin{center}
   \def\svgwidth{207pt}
\begingroup%
  \makeatletter%
  \providecommand\color[2][]{%
    \errmessage{(Inkscape) Color is used for the text in Inkscape, but the package 'color.sty' is not loaded}%
    \renewcommand\color[2][]{}%
  }%
  \providecommand\transparent[1]{%
    \errmessage{(Inkscape) Transparency is used (non-zero) for the text in Inkscape, but the package 'transparent.sty' is not loaded}%
    \renewcommand\transparent[1]{}%
  }%
  \providecommand\rotatebox[2]{#2}%
  \ifx\svgwidth\undefined%
    \setlength{\unitlength}{968.42395242bp}%
    \ifx\svgscale\undefined%
      \relax%
    \else%
      \setlength{\unitlength}{\unitlength * \real{\svgscale}}%
    \fi%
  \else%
    \setlength{\unitlength}{\svgwidth}%
  \fi%
  \global\let\svgwidth\undefined%
  \global\let\svgscale\undefined%
  \makeatother%
  \begin{picture}(1,0.76260738)%
    \put(0,0){\includegraphics[width=\unitlength]{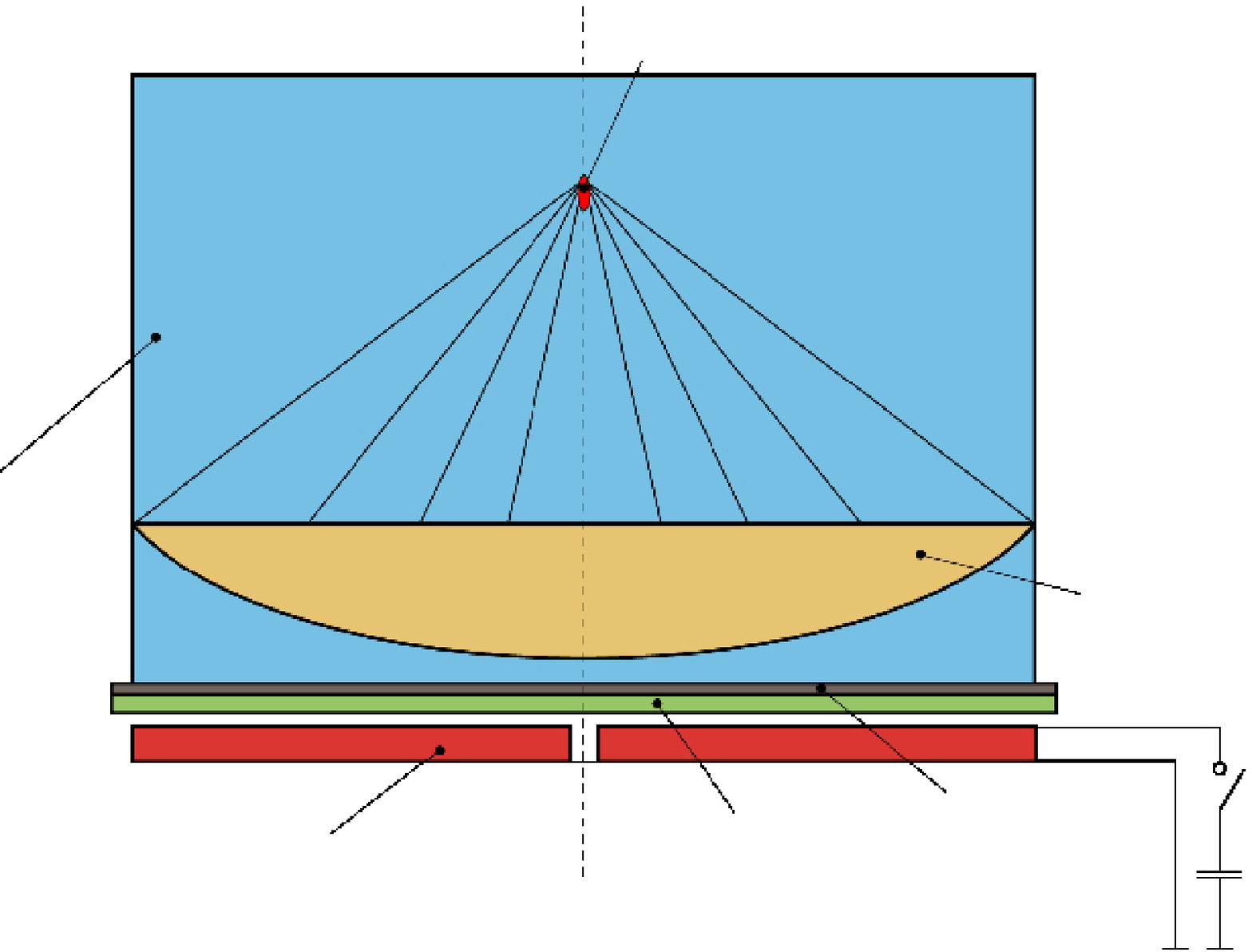}}%
    \put(0.18847806,0.14273979){\color[rgb]{0,0,0}\makebox(0,0)[lt]{\begin{minipage}{0.10739098\unitlength}\raggedright \end{minipage}}}%
    \put(0.2209543,0.07883277){\color[rgb]{0,0,0}\makebox(0,0)[lt]{\begin{minipage}{0.2100615\unitlength}\raggedright \small Coil\end{minipage}}}%
    \put(0.50207329,0.08638679){\color[rgb]{0,0,0}\makebox(0,0)[lt]{\begin{minipage}{0.24074464\unitlength}\raggedright \small Membrane\end{minipage}}}%
    \put(0.72648888,0.1139189){\color[rgb]{0,0,0}\makebox(0,0)[lt]{\begin{minipage}{0.14397472\unitlength}\raggedright \small Rubber\end{minipage}}}%
    \put(0.50080019,0.73600054){\color[rgb]{0,0,0}\makebox(0,0)[lb]{\smash{\small Kidney stone}}}%
    \put(0.28536583,0.79979627){\color[rgb]{0,0,0}\makebox(0,0)[lb]{\smash{\small Axis of Rotation}}}%
    \put(0.87764683,0.26170114){\color[rgb]{0,0,0}\makebox(0,0)[lb]{\smash{\small Lens} $\Omega_+$}}%
    \put(-0.11029899,0.31652083){\color[rgb]{0,0,0}\makebox(0,0)[lb]{\smash{\small Fluid} $\Omega_-$}}%
  \end{picture}%
\endgroup%

 \scriptsize fig. 1: Schematic of a power source in lithotripsy \\ based on the electromagnetic principle
\end{center}
\indent \indent The first goal of the present paper is to show higher interior regularity results for solutions of \eqref{westervelt_damp}. We will show that $u \in H^1(0,T;H_{loc}^2(\Omega))$ and $|\nabla \dot{u}|^{\frac{q-1}{2}}\nabla \dot{u} \in L^2(0,T;H^1_{loc}(\Omega))$. Although $q$-Laplace and parabolic $q$-Laplace equation have been extensively studied in the past (see \cite{lindqvist}, \cite{U}, \cite{Dibendetto}, \cite{Friedman} and references given therein), regularity results in literature on hyperbolic equations with damping of the $q$-Laplace type are sparse and have so far been concerned with local and global well-posedness (see \cite{W}, \cite{Gao}, \cite{BKR13}, \cite{VN}).\\
\indent Secondly, we will consider the coupled problem and show that the solution to \eqref{ModWest_coupled} is piecewise $H^2$ regular in space under the assumption that the gradient of the acoustic pressure remains essentially bounded in space and time. This result is crucial in future numerical approximations of the present problem, as well as in gradient based algorithms for finding the optimal shape of the focusing acoustic lens, where $H^2$ regularity of $u$ is needed in order to express the shape derivative in terms of integrals over the boundary of the lens, see \cite{NK15}.
\subsection{Overview} The paper is organized as follows. We begin in Section \ref{preliminaries} by recalling certain basic results from the theory of finite differences, as well as some helpful inequalities which are needed later. In Section \ref{ModWest_interior_reg} we prove higher interior regularity for the Westervelt equation \eqref{westervelt_damp}. Section \ref{interior_ModWestcoupled} extends these results to the coupled problem. In Section \ref{boundary}, we will show that we can obtain higher regularity up to the boundary of the subdomains if the gradient of the acoustic pressure is essentially bounded in space and time and the subdomains are sufficiently regular.
\section{Preliminaries} \label{preliminaries} \indent In what is to follow, we will need to employ difference quotient approximations to weak derivatives. Assume that $\Omega \subset \mathbb{R}^d$, $d \in \{1,2,3\}$ is an open, connected set with Lipschitz boundary. Let $V \subset \subset \Omega$. $D_r^l$ will stand for the $r$-th difference quotient of size $l$ $$D_r^l u(x,t)=\frac{u(x+l e_r,t)-u(x,t)}{l}, \quad r \in [1,d],$$
for $x \in V$, $l  \in \mathbb{R}$, $ 0<|l|< \frac{1}{2}\text{dist}(V,\partial \Omega)$. Then $D^lu:=(D_1^lu, \ldots,D_d^lu)$. We recall the integration by parts formula for difference quotients $$\int_{V} u \, D_r^l \varphi \, dx =-\int_{V }  D^{-l}_r u \, \varphi \, dx,$$
where $\varphi \in C_c^\infty(V)$, $0<|l|<\frac{1}{2}\text{dist}(V,\partial \Omega)$, as well as the product rule $$D_r^l(\varphi u)=\varphi^l D_r^l u+uD_r^l \varphi,$$
with $\varphi^l(x,t):=\varphi(x+l e_r,t)$. We will also need the following result (cf. Theorem 3, Section 5, \cite{evans}):
\begin{lemma} \label{reg_evans}
\text{\normalfont (a)} Assume $0 \leq q < \infty$ and $u \in W^{1,q+1}(\Omega)$ Then for each $V \subset \subset \Omega$ $$\|D^lu\|_{L^{q+1}(V)} \leq C\|\nabla u\|_{L^{q+1}(\Omega)},$$ for some constant $C$ and all $0<|l|<\frac{1}{2}\text{dist}(V,\partial \Omega)$. \\
 \indent \text{\normalfont (b)} Assume  $0 < q < \infty$, $u \in L^{q+1}(\Omega)$, and there exists a constant $C$ such that $\|D^lu\|_{L^{q+1}(V)} \leq C$ for all  $0<|l|<\frac{1}{2}\text{dist}(V,\partial \Omega)$. Then $$u \in W^{1,q+1}(V), \ \text{with} \ \|\nabla u\|_{L^{q+1}(V)} \leq C.$$
\end{lemma}

\subsection*{Essential inequalities} Before proceeding further, let us also recall several useful inequalities that we will need when handling the $q$-Laplace damping term. They can be found in Chapter 10, \cite{lindqvist} and Appendix, \cite{LiuYan}. From now on, $C_q$ will be used to denote a generic constant depending only on $q$. For any $x,y \in \mathbb{R}^d$ it holds
\begin{align} \label{ineq1}
||x|^{q-1}x-|y|^{q-1}y|\leq C_q |x-y|(|x|+|y|)^{q-1}, \ q>0,
\end{align}
\begin{align} \label{ineq3}
||x|^{q-1}x-|y|^{q-1}y|\geq 2^{-1}|x-y|^{2}(|x|+|y|)^{q-1} \geq 2^{1-q} |x-y|^{q+1} \geq 0, \ q \geq 1.
\end{align}
\begin{align} \label{ineq4}
\frac{4}{(q+1)^2}||x|^{\frac{q-1}{2}}x-|y|^{\frac{q-1}{2}}y|^2\leq  (|x|^{q-1}x-|y|^{q-1}y)\cdot(x-y), \ q\geq 1,
\end{align}
\begin{align} \label{ineq5}
||x|^{q-1}x-|y|^{q-1}y| \leq q(|x|^{\frac{q-1}{2}}+|y|^{\frac{q-1}{2}})
\left| |x|^{\frac{q-1}{2}}-|y|^{\frac{q-1}{2}}\right|, \ q \geq 1.
\end{align}
\noindent We will also need Young's inequality (see for instance Appendix B, \cite{evans}) in the form
\begin{align} \label{Young}
\quad  \quad \quad |xy| \leq \varepsilon |x|^r+C(\varepsilon, r) |y|^{\frac{r}{r-1}}  \quad (\varepsilon >0, \ 1<r< \infty),
\end{align}
with $C(\varepsilon, r)=(r-1)r^{\frac{r}{r-1}}\varepsilon ^{-\frac{1}{1-r}}$. 

\subsection*{Notation}
By $C_{X,Y}^\Omega$ we will denote the norm of the embedding operator $X(\Omega)\to Y(\Omega)$ between two function spaces over the domain $\Omega$.
 
\section{Interior regularity for the Westervelt equation with strong nonlinear damping} \label{ModWest_interior_reg} 
In this section, we will establish higher interior regularity for the equation \eqref{westervelt_damp} with constant coeffcients. Let us first consider the following Dirichlet problem:
\begin{equation}\label{ModWest_Dirichlet}
\begin{cases}
(1-2ku)\ddot{u}-c^2\Delta u-b\,\text{div}\,(((1-\delta) +\delta|\nabla \dot{u}|^{q-1})\nabla \dot{u})
=2k(\dot{u})^2
 \quad \text{ in } \Omega \times (0,T]   ,
 \\
u|_{\partial \Omega} =0 
\quad \text{ for } t \in (0,T],
\\
(u,\dot{u})|_{t=0}=(u_0, u_1)
 \,\, \text{ in } \Omega, \\
\end{cases}
\end{equation}
with the following assumptions on the coefficients and the exponent $q$:
\begin{align} \label{coeff_ModWest_Dirichlet}
c^2, \ b>0,\ \delta \in (0,1), \ k\in \mathbb{R}, \ q>d-1, \ q \geq 1.
\end{align}
The weak formulation reads as 
\begin{equation}\label{ModWest_Dirichlet_weak}
\begin{cases} 
\int_0^T \int_{\Omega} \{(1-2ku)\ddot{u} \phi+c^2\nabla u \cdot \nabla \phi + b(1-\delta)\nabla \dot{u} \cdot \nabla \phi \vspace{1.5mm}\\
\quad \quad \quad +b\delta|\nabla \dot{u}|^{q-1} \nabla \dot{u} \cdot \nabla \phi-2k(\dot{u})^2 \phi\} \, dx \, ds =0 \vspace{1.5mm}\\
\text{holds for all test functions} \ \phi \in \tilde{X}=L^2(0,T;W_0^{1,q+1}(\Omega)),
\end{cases}
\end{equation}
with $(u,\dot{u})=(u_0,u_1)$. We recall the following well-posedness result (cf. Theorem 2.3, \cite{BKR13}):
\begin{proposition}\label{ModWest_Dirichlet_localw} (Local well-posedness)
Let assumptions \eqref {coeff_ModWest_Dirichlet} hold. For any $T>0$ there is a $\kappa_T>0$ such that for all $u_0,u_1\in W_0^{1,q+1}(\Omega)$ with
\[\begin{aligned}
 |u_1|_{L^2(\Omega)}^2 +|\nabla u_0|_{L^2(\Omega)}^2
+ |\nabla u_1|_{L^2(\Omega)}^2
+ |\nabla u_1|_{L^{q+1}(\Omega)}^{q+1}
+|\nabla u_0|_{L^{q+1}(\Omega)}^2\leq \kappa_T^2
\end{aligned}\]
there exists a weak solution $u\in \cW \subset X$ of \eqref{ModWest_Dirichlet}, where $X=H^2(0,T;L^2(\Omega))\cap C^{0,1}(0,T;W_0^{1,q+1}(\Omega))$, and
\begin{eqnarray}\label{defcW}
\cW =\{v\in X 
&:& \|\ddot{v}\|_{L^2(0,T;L_2(\Omega))}\leq \bar{m}
 \wedge \|\nabla \dot{v}\|_{L^\infty(0,T;L^2(\Omega))}\leq \bar{m}\nonumber\\
&& \wedge \|\nabla \dot{v}\|_{L^{q+1}(0,T;L^{q+1}(\Omega))}\leq \bar{M} \wedge  (v,\dot{v})=(u_0,u_1)\}
\end{eqnarray}
with 
\begin{equation}\label{smallnessMbar}
2|k| C_{W_0^{1,q+1},L^\infty}^\Omega
(\kappa_T+ T^{\frac{q}{q+1}} \bar{M}) <1 
\end{equation}
and $\bar{m}$ sufficiently small, and $u$ is unique in $\cW $.
\end{proposition}
In \cite{BKR13}, the issue of possible degeneracy of the Westervelt equation due to the factor $1-2ku$ is resolved by means of the embedding $W_0^{1,q+1}(\Omega) \hookrightarrow L^{\infty}(\Omega)$, valid for $q>d-1$, and the following estimate 
\begin{align*}
|u(x,t)|\leq&\, C^{\Omega}_{W_0^{1,q+1},L^{\infty}}|\nabla u(t)|_{L^{q+1}(\Omega)}\\
\leq& \, C^{\Omega}_{W_0^{1,q+1},L^{\infty}}|\nabla u_0+\int_0^t \nabla \dot{u} \, ds \,|_{L^{q+1}(\Omega)}\\
\leq&\, C^{\Omega}_{W_0^{1,q+1},L^{\infty}}\Bigl(|\nabla u_0|_{L^{q+1}(\Omega)}+\Bigl(t^q \int_0^t \int_\Omega |\nabla \dot{u}(y,s)|^{q+1}\, dy \, ds \Bigr)^{\frac{1}{q+1}}\Bigr),
\end{align*}
which leads to the bound
\begin{equation} \label{degeneracy}
\begin{aligned}
&1-a_0< 1-2ku < 1+a_0, \\
& a_0:= 2|k|C^{\Omega}_{W_0^{1,q+1},L^{\infty}}(|\nabla u_0|_{L^{q+1}(\Omega)}+T^{\frac{q}{q+1}}\|\nabla \dot{u}\|_{L^{q+1}(0,T;L^{q+1}(\Omega))}). 
\end{aligned}
\end{equation}
Due to the embedding $W^{1,q+1}(\Omega) \hookrightarrow C^{0,1-\frac{d}{q+1}}(\overline{\Omega})$, we also know that $u$ is H\" older continuous in space, i.e. $u \in C^{0,1}(0,T;C^{0,1-\frac{d}{q+1}}(\overline{\Omega}))$.
\subsection{Higher interior regularity} We will establish higher interior regularity by following the difference-quotient approach (see, for instance, Theorem 4, Section 6.3.2, \cite{evans}). Let us denote
\begin{align*}
F=|\nabla \dot{u}|^{\frac{q-1}{2}}\nabla \dot{u}, \quad F^l=|\nabla \dot{u}^l|^{\frac{q-1}{2}}\nabla \dot{u}^l.
\end{align*}
As a by-product of the following proof we will also obtain $F \in H^{1}_{loc}(\Omega)$ by adapting the idea of Bojarski and Iwaniec for $q$-harmonic functions (see, for instance, Section 4, \cite{lindqvist}) to our model.
\begin{theorem} \label{thm:interior_reg} (Higher interior regularity) Let assumptions \eqref{coeff_ModWest_Dirichlet} hold true, $u_0 \in H^2(\Omega) \cap W_0^{1,q+1}(\Omega)$, $u_1 \in W_0^{1,q+1}(\Omega)$, and let $u$ be the weak solution of \eqref{ModWest_Dirichlet}. Then $u \in H^1(0,T;H^2_{loc}(\Omega))$ and $|\nabla \dot{u}|^{\frac{q-1}{2}}\nabla \dot{u} \in L^2(0,T;H^1_{loc}(\Omega))$.
\end{theorem}
\begin{proof}
Choose any open set $V \subset \subset \Omega$ and an open set $W$ such that $V \subset \subset W \subset \subset \Omega$. We then introduce a smooth cut-off function $\zeta$ such that
\begin{align*}
\begin{cases}
\zeta=1 \ \text{on} \ V, \ \zeta=0 \ \text{on} \ \Omega \setminus W, \\
0 \leq \zeta \leq 1.
\end{cases}
\end{align*}
Let $|l|>0$ be small and choose $r \in \{1,\ldots, d\}$. We are then allowed to use $$\phi:=-D_r^{-l}(\zeta^2 D_r^l \dot{u})\one_{[0,t)}, \ \ t \in[0,T]$$ as a test function in \eqref{ModWest_Dirichlet_weak}, which results in
\begin{align} \label{interior:est1}
&\frac{1}{2}\Bigl[\int_{\Omega_i}(1-2ku^l)(\zeta D_r^l \dot{u})^2\, dx\Bigr]_0^t+\frac{1}{2}c^2\Bigl[\int_{\Omega}|\zeta D_r^l \nabla u|^2\, dx\Bigr]_0^t \nonumber\\
&+b(1-\delta)\int_0^T \int_{\Omega_i}|\zeta D_r^l \nabla \dot{u}|^2\, dx \, ds 
+\frac{4}{(q+1)^2}b \delta\int_0^t \int_{\Omega}  |\zeta D_r^l F|^{2} \, dx \, ds\nonumber\\ 
\leq& \,2k\int_0^t \int_{\Omega}D_r^l u \, \ddot{u} \,  \zeta^2 D_r^l \dot{u} \, dx \, ds 
 -c^2\int_0^t \int_{\Omega} \zeta \nabla \zeta \cdot D_r^l \nabla u\, D_r^l \dot{u} \, dx \, ds  \nonumber\\
&-2b(1-\delta)\int_0^t \int_{\Omega} \zeta \nabla \zeta \cdot D_r^l \nabla \dot{u}\, D_r^l \dot{u}_i \, dx \, ds \\
&-2b\delta_i \int_0^t \int_{\Omega} D_r^l(|\nabla \dot{u}|^{q-1}\nabla \dot{u}) \cdot \zeta \nabla \zeta\,  D_r^l \dot{u} \, dx \, ds 
+k\int_0^t \int_{\Omega_i}(\dot{u}^l+2\dot{u})(\zeta D_r^l \dot{u})^2 \, dx \, ds. \nonumber
\end{align}
Here we have made use of the estimate
\begin{align*}
\int_0^t \int_{\Omega}\zeta^2 D_r^l(|\nabla \dot{u}|^{q-1}\nabla \dot{u}) D_r^l\nabla \dot{u}  \, dx \, ds \geq \frac{4}{(q+1)^2}\int_0^t \int_{\Omega}  \frac{1}{l^2}\zeta^2||\nabla \dot{u}^l|^{\frac{q-1}{2}}\nabla \dot{u}^l-|\nabla \dot{u}|^{\frac{q-1}{2}}\nabla \dot{u}|^{2}\, dx \, ds,
\end{align*}
which follows from \eqref{ineq4}. Next, we estimate the terms on the right hand side containing $\zeta \nabla \zeta$. We have 
\begin{align*}
& - c^2\int_0^t \int_{\Omega} \zeta \nabla \zeta \cdot D_r^l \nabla u D_r^l \dot{u} \, dx \, ds-2b(1-\delta)\int_0^t \int_{\Omega} \zeta \nabla \zeta \cdot D_r^l \nabla \dot{u} D_r^l \dot{u} \, dx \, ds  \\
\leq& \, C \int_0^T \int_{\Omega} \zeta (|D_r^l \nabla u| + |D_r^l \nabla \dot{u}| )|\nabla \zeta D_r^l \dot{u}|\, dx \, ds \\
\leq& \, \varepsilon T\|\zeta D_r^l \nabla u\|^2_{L^\infty(0,T;L^2(\Omega))}+ \varepsilon\|\zeta D_r^l \nabla \dot{u}\|^2_{L^2(0,T;L^2(\Omega))}+ \frac{C}{\varepsilon}\|\nabla \dot{u}\|^2_{L^2(0,T;L^2(\Omega))},
\end{align*}
where $C$ depends on $c, b, \delta$ and $|\nabla \zeta|_{L^\infty(W)}$ and we have used Lemma \ref{reg_evans}, (a). By employing estimate \eqref{ineq5} and H\" older's inequality we obtain
\begin{align*}
&-2b\delta \int_0^t \int_{\Omega} D_r^l(|\nabla \dot{u}|^{q-1}\nabla \dot{u})\cdot \zeta \nabla \zeta \, D_r^l \dot{u} \, dx \, ds \\
\leq& \, 2 b \delta \int_0^T \int_{\Omega} \frac{1}{l}|D_r^l \dot{u}|
\left|
|\nabla \dot{u}^l|^{q-1}\nabla \dot{u}^l-|\nabla \dot{u}|^{q-1}\nabla \dot{u}
\right|
\zeta|\nabla \zeta|\, dx \, ds \\
\leq&\, 2 q b \delta \int_0^T \int_{\Omega}|\nabla \zeta D_r^l \dot{u}|(|\nabla \dot{u}^l|^{\frac{q-1}{2}}+|\nabla \dot{u}|^{\frac{q-1}{2}})|\zeta D_r^l F| \, dx \, ds \\
\leq&\, 2 q b \delta \Bigl\{ \int_0^T \int_{\Omega}|\nabla \zeta D_r^l \dot{u}|^{q+1}\, dx \, ds \Bigr \}^{\frac{1}{q+1}}\Bigl\{ \int_0^T \int_{\text{supp} \zeta} (|\nabla \dot{u}^l|^{\frac{q-1}{2}}+|\nabla \dot{u}|^{\frac{q-1}{2}})^{\frac{2(q+1)}{q-1}}\, dx\, ds\Bigr\}^{\frac{q-1}{2(q+1)}}\\
& \quad \quad \quad \times \Bigl\{\int_0^T \int_{\Omega}|\zeta D_r^l F|^2 \, dx \, ds \Bigr\}^{1/2}.
\end{align*}
The second integral can be majorized with the help of Minkowski's inequality by 
\begin{align*}
&\Bigl\{ \int_0^T \int_{\text{supp} \zeta} (|\nabla \dot{u}^l|^{\frac{q-1}{2}}+|\nabla \dot{u}|^{\frac{q-1}{2}})^{\frac{2(q+1)}{q-1}}\, dx\, ds\Bigr\}^{\frac{q-1}{2(q+1)}} \\
\leq&\, \Bigl\{ \int_0^T \int_{\text{supp} \zeta} (|\nabla \dot{u}^l|^{q+1} \, dx \, ds\Bigr\}^{\frac{q-1}{2(q+1)}} + \Bigl\{ \int_0^T \int_{\text{supp} \zeta} (|\nabla \dot{u}|^{q+1} \, dx \, ds\Bigr\}^{\frac{q-1}{2(q+1)}} \\
\leq&\, 2  \Bigl\{ \int_0^T \int_{\Omega} (|\nabla \dot{u}|^{q+1} \, dx \, ds\Bigr\}^{\frac{q-1}{2(q+1)}}=2\|\nabla \dot{u}\|^{\frac{q-1}{2}}_{L^{q+1}(0,T;L^{q+1}(\Omega))},
\end{align*}
for small $|l|$. Utilizing Young's inequality  then yields 
\begin{align*}
&-2b\delta \int_0^t \int_{\Omega_i} D_r^l(|\nabla \dot{u}|^{q-1}\nabla \dot{u})\cdot \zeta \nabla \zeta \, D_r^l \dot{u} \, dx \, ds \\
\leq&\, b\delta \varepsilon \int_0^T \int_{\Omega}|\zeta D_r^l F|^2 \, dx \, ds +C(\|\nabla \dot{u}\|^{q+1}_{L^{q+1}(0,T;L^{q+1}(W))}+\|D_r^l \dot{u}\|^{q+1}_{L^{q+1}(0,T;L^{q+1}(W))}),
\end{align*}
where $C>0$ depends on $b_i$, $\delta_i$, $\varepsilon$, $q$ and $|\nabla \zeta|_{L^{\infty}(W)}$. Note that the first term in the last line can be absorbed by the $b_i\delta_i$- term on the left hand side in \eqref{interior:est1} for sufficiently small $\varepsilon>0$. The two remaining terms on the right hand side in \eqref{interior:est1} can be estimated as follows
\begin{align*}
&k\int_0^t \int_{\Omega}(\dot{u}^l+2\dot{u})\zeta^2(D_r^l \dot{u})^2 \, dx \, ds+2k\int_0^t \int_{\Omega}\zeta^2 \ddot{u}\, D_r^l u \, D_r^l \dot{u} \, dx \, ds\\
\leq& \,C\Bigl(\|\dot{u}\|_{L^\infty(0,T;L^\infty(\Omega))}\|\zeta D_r^l \dot{u}\|^2_{L^2(0,T;L^2(W))}\\
&+(C_{H^1,L^4}^{\Omega})^2\|\ddot{u}\|_{L^2(0,T;L^2(\Omega))}\|\zeta D_r^l u\|_{L^\infty(0,T;H^1(\Omega))}\|\zeta D_r^l \dot{u}\|_{L^2(0,T;H^1(\Omega))}\Bigr) \\
\leq& C\Bigl(\|\dot{u} \|_{L^\infty(0,T;L^\infty(\Omega))}\|D_r^l \dot{u}\|^2_{L^2(0,T;L^2(\Omega))}\\
&+\bar{m}^2\frac{1}{4\varepsilon}(\|D_r^l u\|^2_{L^\infty(0,T;L^2(W))}+\|\zeta D_r^l \nabla u\|^2_{L^\infty(0,T;L^2(\Omega))}) \\
&+\varepsilon\|D_r^l \dot{u}\|^2_{L^2(0,T;L^2(W))}+\varepsilon\|\zeta D_r^l \nabla \dot{u}\|^2_{L^2(0,T;L^2(\Omega))}\Bigr)\, .
\end{align*}
Altogether, for sufficiently small $\varepsilon>0$ and $\bar{m}$, we can achieve that
\begin{align*}
&\|\zeta D_r^l \dot{u}\|^2_{L^\infty(0,T;L^2(\Omega))}+\|\zeta D_r^l \nabla u_i\|^2_{L^\infty(0,T;L^2(\Omega_i))}+\|\zeta D_r^l \nabla \dot{u}\|^2_{L^2(0,T;L^2(\Omega))}\\
&+\|\zeta D_r^l F\|^{q+1}_{L^{2}(0,T;L^{2}(\Omega))}\\
\leq& \, C (\|\nabla \dot{u}\|^{q+1}_{L^{\infty}(0,T;L^\infty(\Omega))}+\|D_r^l \dot{u}\|^{q+1}_{L^{q+1}(0,T;L^{q+1}(W))}+\|\dot{u}\|_{L^\infty(0,T;L^\infty(\Omega))}\|D_r^l \dot{u}\|^2_{L^2(0,T;L^2(W))}\\
&+\|D_r^l u\|^2_{L^\infty(0,T;L^2(W))}+\|D_r^l \dot{u}\|^2_{L^2(0,T;L^2(W))}+|D_r^l u_1|^2_{L^2(\Omega)}+|D_r^l \nabla u_0|^2_{L^2(\Omega)}).
\end{align*}
By remembering the definition of $\zeta$ and Lemma \ref{reg_evans}, we finally arrive at
\begin{align*}
&\| D_r^l \dot{u}_i\|^2_{L^\infty(0,T;L^2(V))}+\| D_r^l \nabla u\|^2_{L^\infty(0,T;L^2(V))}
+\|D_r^l \nabla \dot{u}\|^2_{L^2(0,T;L^2(V))}+\|D_r^l F\|^{2}_{L^{2}(0,T;L^{2}(V))}\\
\leq& \, C (\|\nabla \dot{u}\|^{q+1}_{L^{\infty}(0,T;L^\infty(\Omega))}+(1+\|\dot{u} \|_{L^\infty(0,T;L^\infty(\Omega))})\|\nabla \dot{u}\|^2_{L^2(0,T;L^2(\Omega))}+\|\nabla u\|^2_{L^\infty(0,T;L^2(\Omega))}\\
&+|\nabla u_1|^2_{L^2(\Omega)}+|u_0|^2_{H^2(\Omega)}),
\end{align*}
for $r \in [1,d]$, sufficiently small $|l|>0 $ and sufficiently large $C>0$ which does not depend on $l$. By employing Lemma \ref{reg_evans}, we can conclude that $u \in H^{1}(0,T;H^{2}_{loc}(\Omega))$ and $|\nabla \dot{u}|^{\frac{q-1}{2}}\nabla \dot{u} \in L^2(0,T;H^1_{loc}(\Omega))$.
\end{proof}
\indent As a simple consequence of the previous proposition, we can obtain H\" older continuity of $u$ (see Section 4, \cite{lindqvist}). Indeed, since $F \in L^2(0,T;H^1_{loc}(\Omega))$ and $d \in \{1,2,3\}$, due to Sobolev's embedding theorem we have that $F \in L^2(0,T;L^{6}_{loc}(\Omega))$. This implies that $u \in W^{1,q+1}(0,T;W^{1,3(q+1)}_{loc}(\Omega))$. We can than conclude that $u \in W^{1,q+1}(0,T;C^{0,\alpha}_{loc}(\Omega))$, where $\alpha=1-\frac{d}{3(q+1)}$. \\
\indent 
When $d \in  \{1,2\}$ we can do even better. According to Sobolev's embedding theorem, $u \in W^{1,q+1}(0,T;C_{loc}^{1,\frac{1}{2}}(\Omega))$ if $d=1$, and $u \in W^{1,q+1}(0,T;C_{loc}^{0,\gamma}(\Omega))$ if $d=2$, where $\gamma \in (0,1)$. Altogether, we have
\begin{corollary} Let the assumptions of Theorem \ref{thm:interior_reg} hold true. Then 
\begin{align} \label{H_reg}
u \in \begin{cases}
 W^{1,q+1}(0,T;C^{0,1-\frac{1}{q+1}}_{\text{loc}}(\Omega)) \ \text{if} \ d=3, \vspace{1mm} \\
 W^{1,q+1}(0,T;C^{0,\gamma}_{\text{loc}}(\Omega)) \ \text{with} \ \gamma \in (0,1), \ \text{if} \ d=2, \vspace{1mm} \\
 W^{1,q+1}(0,T;C^{1,\frac{1}{2}}_{\text{loc}}(\Omega)) \ \text{if} \ d=1.
 \end{cases}
\end{align}
\end{corollary}
\subsection{Neumann problem for the Westervelt equation} Let us also consider the Neumann problem for the Westervelt equation with strong nonlinear damping:
\begin{align} \label{ModWest_Neumann}
\begin{cases}
(1-2ku)\ddot{u}-c^2\Delta u-b\,\text{div}(((1-\delta) +\delta|\nabla \dot{u}|^{q-1})\nabla \dot{u}) 
 =2k(\dot{u})^2 
 \, \text{ in } \Omega \times (0,T],
\vspace{2mm} \\
c^2 \frac{\partial u}{\partial n}+b((1-\delta)+\delta|\nabla \dot{u}|^{q-1})\frac{\partial \dot{u}}{\partial n}=g  \ \ \text{on} \ \partial\Omega \times (0,T],
\vspace{2mm} \\
(u,\dot{u})\vert_{t=0}=(u_0, u_1) \ \ \text{on} \ \overline{\Omega},
,\\
\end{cases}
\end{align}
with the same assumptions \eqref{coeff_ModWest_Dirichlet} on coefficients. Problem \eqref{ModWest_Neumann} is locally well-posed thanks to the following result (cf. Theorem 2.5, \cite{VN}):
\begin{proposition} \label{thm:W1}
Let $g \in L^{\infty }(0,T;W^{-\frac{q}{q+1},\frac{q+1}{q}}(\partial \Omega))$, $\dot{g} \in L^{\frac{q+1}{q}}(0,T;W^{-\frac{q}{q+1},\frac{q+1}{q}}(\partial \Omega))$, and $u_0, u_1 \in W^{1,q+1}(\Omega)$. 
For sufficiently small initial and boundary data, final time $T$ and $\bar{m}$ and $\bar{M}$ there exists a unique weak solution $u \in \cW \subset X$ of \eqref{ModWest_Neumann}, where $X=H^2(0,T;L^2(\Omega))\cap C^{0,1}(0,T;W^{1,q+1}(\Omega))$, and
\begin{equation}\label{defcW1}
\begin{split}
\cW =\{v\in X 
:& \ \|\ddot{v}\|_{L^2(0,T;L^2(\Omega))}\leq \bar{m}  \wedge \| \dot{v}\|_{L^{\infty}(0,T;H^1(\Omega))}\leq \bar{m}\\
& \wedge \| \nabla \dot{v}\|_{L^{q+1}(0,T;L^{q+1}(\Omega))}\leq \bar{M} 
\}.
\end{split}
\end{equation}
\end{proposition}
\noindent By inspecting the proof of Proposition \ref{thm:interior_reg}, we immediately obtain higher interior regularity result for the present model, since the cut-off function used in the proof vanishes near the boundary:
\begin{corollary} \label{interior_ModWest_Neumann}  Let the assumptions  \eqref{coeff_ModWest_Dirichlet} hold true, $u_0 \in H^2(\Omega) \cap W^{1,q+1}(\Omega)$, $u_1 \in W^{1,q+1}(\Omega)$, and let $u$ be the weak solution of \eqref{ModWest_Neumann}. Then $u \in H^1(0,T;H^2_{loc}(\Omega))$ and $|\nabla \dot{u}|^{\frac{q-1}{2}}\nabla \dot{u} \in L^2(0,T;H^1_{loc}(\Omega))$. Moreover, \eqref{H_reg} holds.
\end{corollary}
\section{Interior regularity for the coupled problem} \label{interior_ModWestcoupled}
Let us now assume that $\Omega \subset \mathbb{R}^d$, $d \in \{1,2,3\}$, is a bounded domain with Lipschitz  boundary $\partial \Omega$, and $\Omega_{+}$ a subdomain, representing the lens, such that $\bar{\Omega}_{+} \subset \Omega$ and $\Omega_{+}$  has Lipschitz boundary $\partial \Omega_+=\Gamma$. 
\begin{minipage}{0.4\textwidth}
\begin{center}
\vspace{0.2cm}
 \def\svgwidth{130pt}
\begingroup%
  \makeatletter%
  \providecommand\color[2][]{%
    \errmessage{(Inkscape) Color is used for the text in Inkscape, but the package 'color.sty' is not loaded}%
    \renewcommand\color[2][]{}%
  }%
  \providecommand\transparent[1]{%
    \errmessage{(Inkscape) Transparency is used (non-zero) for the text in Inkscape, but the package 'transparent.sty' is not loaded}%
    \renewcommand\transparent[1]{}%
  }%
  \providecommand\rotatebox[2]{#2}%
  \ifx\svgwidth\undefined%
    \setlength{\unitlength}{321.05312653bp}%
    \ifx\svgscale\undefined%
      \relax%
    \else%
      \setlength{\unitlength}{\unitlength * \real{\svgscale}}%
    \fi%
  \else%
    \setlength{\unitlength}{\svgwidth}%
  \fi%
  \global\let\svgwidth\undefined%
  \global\let\svgscale\undefined%
  \makeatother%
  \begin{picture}(1,0.88560109)%
    \put(0.15,0.15){\includegraphics[width=0.7\unitlength]{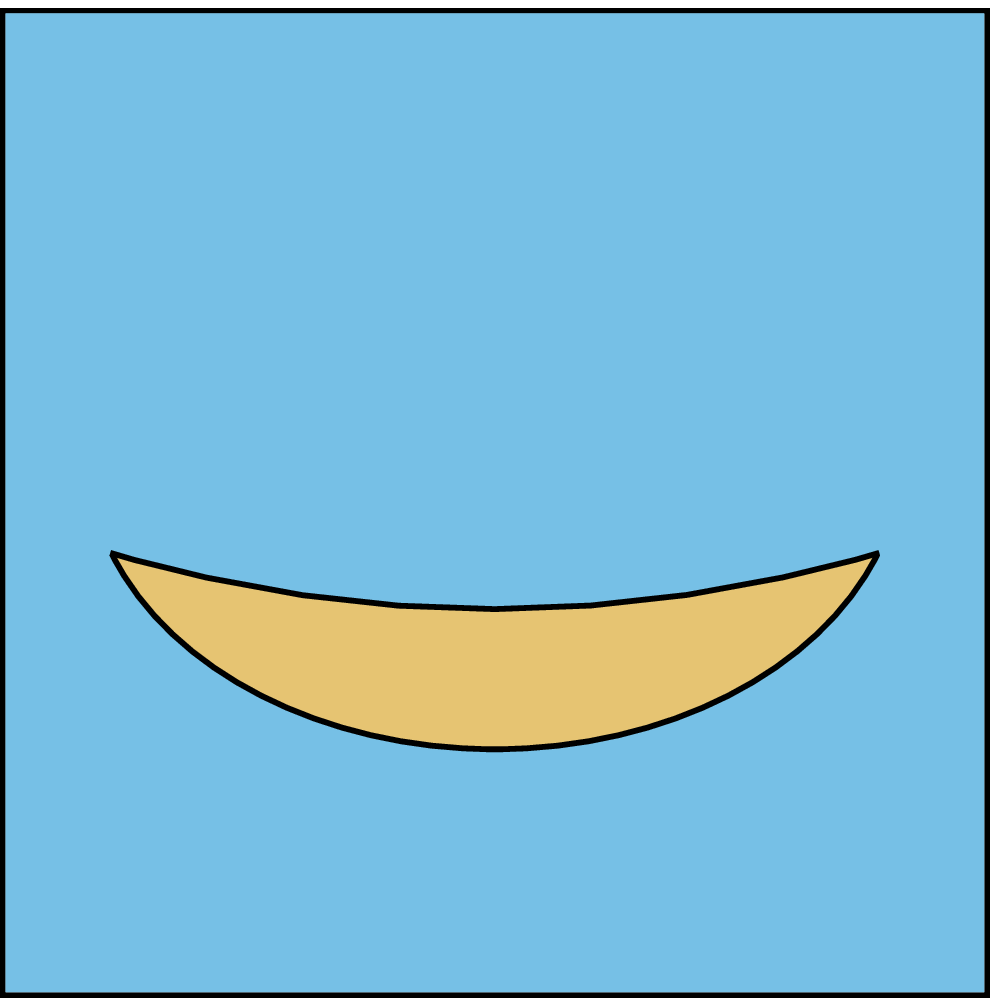}}%
    \put(0.45408068,0.33151276){\color[rgb]{0,0,0}\makebox(0,0)[lb]{$\Omega_+$}}%
    \put(0.6054314,0.64808884){\color[rgb]{0,0,0}\makebox(0,0)[lb]{$\Omega_-$}}%
    \put(0.31140779,0.46382592){\color[rgb]{0,0,0}\makebox(0,0)[lb]{$\Gamma$}}%
    \put(0.05172079,0.49225733){\color[rgb]{0,0,0}\makebox(0,0)[lb]{$\Omega$}}%
  \end{picture}%
\endgroup%

\vspace{-0.7cm}
\\ \scriptsize Lens $\Omega_+$ and fluid $\Omega_-$ regions \\
\end{center}
\end{minipage}
\begin{minipage}{0.59\textwidth}
We denote by $\Omega_{-}=\Omega \setminus \bar{\Omega}_{+}$ the part of the domain representing the fluid region. We then have  $\partial \Omega_-=\Gamma \cup \partial \Omega$. \\
$n_{+}$, $n_{-}$  will stand for the unit outer normals to lens $\Omega_{+}$ and fluid region $\Omega_{-}$. Restrictions of a function $v$ to $\Omega_{+,-}$ will be denoted by $v_{+}$, $v  _{-}$ and $\llbracket v \rrbracket:=v_+-v_-$ will denote the jump over $\Gamma$.\\
Note that the assumption on the regularity of the subdomains will be strengthened to $C^{1,1}$ when showing higher regularity up to the boundary of the subdomains.\\
\end{minipage} \vspace{1mm}\\
The coefficients in \eqref{ModWest_coupled} will only be allowed to jump over the interface $\Gamma$:
\begin{equation} \label{coeff}
\begin{cases}
 b,\varrho, \lambda, \delta, k \in L^{\infty}(\Omega), \\
b_i:=b\vert_{\Omega_i},   \varrho_i:=\varrho\vert_{\Omega_i},  \lambda_i:=\lambda\vert_{\Omega_i} >0, \ \delta_i:=\delta\vert_{\Omega_i} \in (0,1), \ k_i:=k\vert_{\Omega_i} \in \mathbb{R} \quad \text{for} \ i \in \{+,-\}.
\end{cases}
\end{equation}
We assume that $q \geq 1$, $q>d-1$.
The strong formulation of \eqref{ModWest_coupled} reads as follows:
\begin{align} \label{ModWest_coupled_Dirichlet}
\begin{cases}
\frac{1}{\lambda(x)}(1-2k(x)u)\ddot{u}-\div(\frac{1}{\varrho(x)}\nabla u)
-\text{div}\Bigl(b(x)((1-\delta(x))+\delta(x)|\nabla \dot{u}|^{q-1})\nabla \dot{u}\Bigr) \vspace{1.5mm} \\
=\frac{2k(x)}{\lambda(x)}(\dot{u})^2 \quad \text{ in } \Omega_{+} \cup \Omega_{-}, \vspace{1.5mm} \\
 \llbracket u \rrbracket =0 \quad \text{on} \ \Gamma=\partial \Omega_+, \vspace{1.5mm} \\
  \Bigl \llbracket \frac{1}{\varrho} \frac{\partial u}{\partial n_{+}}+b(1-\delta)\frac{\partial \dot{u}}{\partial n_{+}}+b\delta|\nabla \dot{u}|^{q-1}\frac{\partial \dot{u}}{\partial n_{+}}\Bigr \rrbracket =0 \quad \text{on} \ \Gamma=\partial \Omega_+, \vspace{1mm} \\
 u=0 \quad \text{on} \ \partial \Omega, \vspace{1mm}\\
(u,\dot{u})\vert_{t=0}=(u_0,u_1).
\end{cases}
\end{align}
This model was studied (in an equivalent one domain formulation) in \cite{BKR13}. We will utilize the following well-posedness result (cf. Theorem 2.3 and Corollary 4.1, \cite{BKR13}):
\begin{proposition} (Local well-posedness)\label{aa_corollary}
Let $q>d-1$, $q \geq 1$ and the assumptions \eqref{coeff} hold. For any $T>0$ there is a $\kappa_T>0$ such that for all  $u_0, u_1 \in W_0^{1,q+1}(\Omega)$ with
\begin{align*}
|u_1|^2_{L^2(\Omega)}+|\nabla u_0|^2_{L^2(\Omega)}+|\nabla u_1|^2_{L^2(\Omega)}+|\nabla u_0|^2_{L^{q+1}(\Omega)}+|\nabla u_1|^{q+1}_{L^{q+1}(\Omega)} \leq \kappa_T^2
\end{align*}
there exists a unique solution $u\in \cW \subset X$ of \eqref{ModWest_coupled_Dirichlet}, where $X=H^2(0,T;L^2(\Omega))\cap C^{0,1}(0,T;W_0^{1,q+1}(\Omega))$, and
\begin{eqnarray} 
\cW =\{v\in X
&:& \|\ddot{v}\|_{L^2(0,T;L^2(\Omega))}\leq \bar{m}  
 \wedge \|\nabla \dot{v}\|_{L^\infty(0,T;L^2(\Omega))}\leq \bar{m} \\
&& \wedge \|\nabla \dot{v}\|_{L^{q+1}(0,T;L^{q+1}(\Omega))}\leq \bar{M} \wedge  (v,\dot{v})=(u_0,u_1) \},\nonumber 
\end{eqnarray}
with \begin{equation*}
2\overline{k}C_{W_0^{1,q+1},L^\infty}^\Omega(\kappa_T+
T^{\frac{q}{q+1}} \bar{M}) <1, 
\end{equation*}  and $\bar{m}$ sufficiently small.
\end{proposition}
\noindent For simplicity of exposition, higher interior and later boundary regularity will be obtained under the assumption that the coefficients in \eqref{ModWest_coupled_Dirichlet} are piecewise costant functions, i.e.
 \begin{equation} \label{coeff_const}
\begin{cases}
b_i:=b\vert_{\Omega_i},\varrho_i:=\varrho\vert_{\Omega_i}, \lambda_i:=\lambda\vert_{\Omega_i}, \delta_i:=\delta\vert_{\Omega_i}, k_i:=k\vert_{\Omega_i} \ \text{are constants}, \\
b_i,   \varrho_i,  \lambda_i >0, \quad \delta_i \in (0,1), \quad k_i \in \mathbb{R} \quad \text{for} \ i \in \{+,-\}.
\end{cases}
\end{equation}
We denote  $\underline{\omega}=\min\{|\omega_+|,|\omega_-|\}$, $\overline{\omega}=\max\{|\omega_+|,|\omega_-|\}$, where $\omega \in \{b, \varrho, \lambda, \delta, k\}$.  
The proof of Theorem \ref{thm:interior_reg} can be carried over in a straightforward manner to the coupled problem to show higher interior regularity within each of the subdomains:
\begin{corollary} \label{interior_reg_coupled}  Assume that $q\geq 1$, $q>d-1$, $u_0\vert_{\Omega_i} \in H^2(\Omega_i)$, $i\in\{+,-\}$, $u_0,u_1 \in W_0^{1,q+1}(\Omega)$ and assumptions \eqref{coeff_const} on the coefficients hold true. Let $u$ be the weak solution of \eqref{ModWest_coupled_Dirichlet}. Then $u_i \in H^1(0,T;H^2_{loc}(\Omega_i))$ and $|\nabla \dot{u}_i|^{\frac{q-1}{2}}\nabla \dot{u}_i \in L^2(0,T;H^1_{loc}(\Omega_i))$, $i \in \{+,-\}$. Moreover, \eqref{H_reg} holds with $u$ replaced by $u_i$ and $\Omega$ by $\Omega_i$, $i \in \{+,-\}$.
\end{corollary}
\subsection{Neumann problem for the coupled system}
Let us also consider the coupled problem with Neumann boundary conditions on the outer boundary of the fluid subdomain:
\begin{align} \label{coupled_Neumann}
\begin{cases}
\frac{1}{\lambda(x)}(1-2k(x)u)\ddot{u}-\div(\frac{1}{\varrho(x)}\nabla u)
-\text{div}(b(x)((1-\delta(x))+\delta(x)|\nabla \dot{u}|^{q-1})\nabla \dot{u}) \vspace{1.5mm} \\
=\frac{2k(x)}{\lambda(x)}(\dot{u})^2 \quad \text{ in } \Omega_{+} \cup \Omega_{-}, \vspace{1.5mm} \\
 [u]=0 \quad \text{on} \ \Gamma=\partial \Omega_+, \vspace{1.5mm} \\
  \Bigl[\frac{1}{\varrho} \frac{\partial u}{\partial n_{+}}+b(1-\delta)\frac{\partial \dot{u}}{\partial n_{+}}+b\delta|\nabla \dot{u}|^{q-1}\frac{\partial \dot{u}}{\partial n_{+}}\Bigr]=0 \quad \text{on} \ \Gamma=\partial \Omega_+, \vspace{1.5mm} \\
\frac{1}{\varrho} \frac{\partial u}{\partial n_{+}}+b(1-\delta)\frac{\partial \dot{u}}{\partial n_{+}}+b\delta|\nabla \dot{u}|^{q-1}\frac{\partial \dot{u}}{\partial n_{+}}=g \quad \text{on} \ \partial \Omega, \vspace{1.5mm}\\
(u,\dot{u})\vert_{t=0}=(u_0,u_1).
\end{cases}
\end{align}
We will show that the higher regularity result is valid for this model as well. The weak form of the problem is given as follows:
\begin{align*} 
&\int_0^T \int_\Omega \Bigl\{\frac{1}{\lambda}(1-2ku) \ddot{u}\phi + \frac{1}{\varrho} \nabla u \cdot \nabla \phi 
+ b((1-\delta) +\delta |\nabla \dot{u}|^{q-1})\nabla \dot{u} \cdot \nabla \phi - \frac{2k}{\lambda}(\dot{u})^2 \phi \Bigr\} \, dx \, ds  \\
=& \, \int_0^T \int_{\partial \Omega}g \phi \, dx \, ds, 
\end{align*}
for all $\phi\in L^2(0,T;W^{1,q+1}(\Omega))$, with initial conditions $(u_{0},u_{1})$. We recall the following well-posedness result (cf. \cite{VN}):
\begin{proposition} \label{prop:Neumann}
Let assumptions \eqref{coeff} hold. Let $q>d-1$, $q \geq 1$, $g \in L^{\infty }(0,T;W^{-\frac{q}{q+1},\frac{q+1}{q}}(\partial \Omega))$, $\dot{g} \in L^{\frac{q+1}{q}}(0,T;W^{-\frac{q}{q+1},\frac{q+1}{q}}(\partial \Omega))$, and $u_0, u_1 \in W^{1,q+1}(\Omega)$. 
For sufficiently small initial and boundary data and final time $T$, there exists a unique weak solution $u \in \cW \subset X$ of \eqref{coupled_Neumann}, where $X= H^2(0,T;L^2(\Omega)) \cap C^{0,1}(0,T;W^{1,q+1}(\Omega))$ and
\begin{equation} 
\begin{split}
\cW =\{v\in X 
:& \ \|\ddot{v}\|_{L^2(0,T;L^2(\Omega))}\leq \bar{m}  \wedge \|\dot{v}\|_{L^{\infty}(0,T;H^1(\Omega))}\leq \bar{m}\\
& \wedge \| \nabla \dot{v}\|_{L^{q+1}(0,T;L^{q+1}(\Omega))}\leq \bar{M} 
\}.
\end{split}
\end{equation}
\end{proposition}
\noindent The interior regularity result can again be transferred from Corollary \ref{interior_reg_coupled} to the present model:
\begin{corollary} \label{interior_reg_Neumann}  Let the assumptions of Proposition \ref{prop:Neumann} and assumptions  \eqref{coeff_const} on the coefficients hold true, let 
$u_0\vert_{\Omega_i} \in H^2(\Omega_i)$, $i\in\{+,-\}$, $u_0,u_1 \in W^{1,q+1}(\Omega)$, and let $u$ be the weak solution of \eqref{coupled_Neumann}. Then $u_i \in H^1(0,T;H^2_{loc}(\Omega_i))$ and $|\nabla \dot{u}_i|^{\frac{q-1}{2}}\nabla \dot{u}_i \in L^2(0,T;H^1_{loc}(\Omega_i))$, $i \in \{+,-\}$. Furthermore, \eqref{H_reg} holds with $u$ replaced by $u_i$ and $\Omega$ by $\Omega_i$, $i \in \{+,-\}$.
\end{corollary}
\section{Boundary regularity for the coupled problem} \label{boundary}
We  will show next that the $H^2$-regularity result can be extended up to the boundary of each of the subdomains under the assumption that the gradient of $u$ is essentially bounded in time and space on the whole domain. For this property to hold, we will need to smoothen out the subdomains, i.e. assume that they are $C^{1,1}$ regular. The proof will expand on the approach taken in Lemma 3.6, \cite{Ammari}.
\begin{theorem} \label{thm:boundary_reg_coupled_Dirichlet}(Boundary $H^2$-regularity)
Let the assumptions of Corollary \ref{interior_reg_coupled} hold and let $\partial \Omega$ and $\Gamma=\partial \Omega_+$ be $C^{1,1}$ regular. If $u \in W^{1,\infty}(0,T;W^{1,\infty}(\Omega))$ and $\|\nabla \dot{u}\|_{L^{\infty}(\Omega)(0,T;L^\infty(\Omega))}$ is sufficiently small, then $u_{+,-} \in H^{1}(0,T;H^{2}(\Omega_{+,-}))$.
\end{theorem}
\begin{proof}
We will only show that $u_+ \in H^{1}(0,T;H^{2}(\Omega_+))$, since $u_- \in H^{1}(0,T;H^{2}(\Omega_-))$ follows analogously. \\

\noindent \textit{Step 1: Straightening the boundary.} We begin by straightening the boundary through the change of coordinates near a boundary point (cf. Theorem 4, Section 6.3.2, \cite{evans}). Choose any point $x_0 \in \partial \Omega_+$. There exists a ball $B=B_r(x_0)$ for some $r>0$ and a $C^{1,1}$-diffeomorphism $\Psi:B \rightarrow \Psi(B) \subset \mathbb{R}^d$ such that $\text{det}|\nabla \Psi|=1$, $U^{\prime}=\Psi(B)$ is an open set, $\Psi(B \cap \Omega_+) \subset \mathbb{R}^d_+$ and $\Psi(B\cap \Gamma) \subset \partial \mathbb{R}^d_+$, where $\mathbb{R}^d_+$ is the half-space in the new coordinates.
\vspace{2.5mm} \\
\begin{center}
   \def\svgwidth{207pt}
\begingroup%
  \makeatletter%
  \providecommand\color[2][]{%
    \errmessage{(Inkscape) Color is used for the text in Inkscape, but the package 'color.sty' is not loaded}%
    \renewcommand\color[2][]{}%
  }%
  \providecommand\transparent[1]{%
    \errmessage{(Inkscape) Transparency is used (non-zero) for the text in Inkscape, but the package 'transparent.sty' is not loaded}%
    \renewcommand\transparent[1]{}%
  }%
  \providecommand\rotatebox[2]{#2}%
  \ifx\svgwidth\undefined%
    \setlength{\unitlength}{500.0000051bp}%
    \ifx\svgscale\undefined%
      \relax%
    \else%
      \setlength{\unitlength}{\unitlength * \real{\svgscale}}%
    \fi%
  \else%
    \setlength{\unitlength}{\svgwidth}%
  \fi%
  \global\let\svgwidth\undefined%
  \global\let\svgscale\undefined%
  \makeatother%
  \begin{picture}(1,0.35411162)%
    \put(0,0){\includegraphics[width=\unitlength]{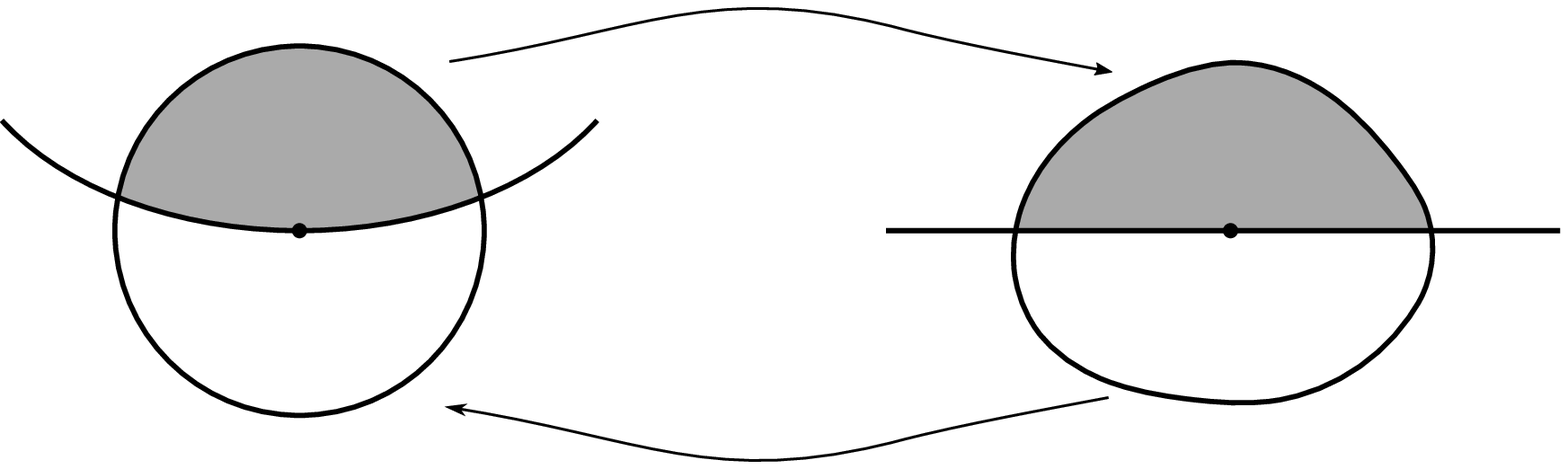}}%
    \put(0.17137779,0.1){\color[rgb]{0,0,0}\makebox(0,0)[lb]{\smash{$x_0$}}}%
    \put(0.7701333,0.1){\color[rgb]{0,0,0}\makebox(0,0)[lb]{\smash{$y_0$}}}%
    \put(0.03906251,-0.015){\color[rgb]{0,0,0}\makebox(0,0)[lb]{\smash{x-coordinates}}}%
    \put(0.67881802,-0.015){\color[rgb]{0,0,0}\makebox(0,0)[lb]{\smash{y-coordinates}}}%
    \put(0.47368677,0.32068974){\color[rgb]{0,0,0}\makebox(0,0)[lb]{\smash{$\Phi$}}}%
    \put(0.46980093,0.01){\color[rgb]{0,0,0}\makebox(0,0)[lb]{\smash{$\Psi$}}}%
  \end{picture}%
\endgroup%
 
 \\[1ex] {\scriptsize straightening out the boundary}
\end{center}
We change the variables and write
\begin{align*}
&y=\Psi(x), \ x \in B, \\
&x=\Phi(y), \ y \in U^{\prime}.
\end{align*} 
Then we have $\Psi(B \cap \Omega_+)=\{y \in U^{\prime}: y_n>0\}$. We denote
\begin{align*}
B^+=B_{\frac{r}{2}}\cap \Omega_+, \ G=\Psi(B_{\frac{r}{2}}(x_0)), \ G^+=\Psi(B^+).
\end{align*}
Then $G \subset \subset U^{\prime}$ and $G^{+} \subset G$. We define
\begin{align*}
w(y,t):= u(\Phi(y),t), \ (y,t) \in U^{\prime} \times [0,T].
\end{align*}
It immediately follows that $w(t):=w(\cdot,t) \in W^{1,q+1}(U^\prime)$. We now transform the original equation on $B \times [0,T]$ into an equation on $U^{\prime} \times [0,T]$:
\begin{equation} \label{eq:boundary_reg}
\begin{aligned}
& \int_{U^{\prime}}\Bigl\{\frac{1}{\hat{\lambda}}(1-2\hat{k}w(t))\ddot{w}(t)\phi+\displaystyle \sum_{i,j=1}^d\Bigl(\hat{\sigma}_{ij}D_iw(t)D_j\phi+\hat{\xi}_{ij}D_i \dot{w}_i(t) D_j \phi \\
&+|J_\Phi^T\nabla \dot{w}(t)|^{q-1}\hat{\eta}_{ij}D_i\dot{w}(t)D_j  \phi\Bigr)-\frac{2\hat{k}}{\hat{\lambda}}(\dot{w}(t))^2\phi\Bigr\} \, dy=0,
\end{aligned}
\end{equation}
for a.e. $t \in [0,T]$, and all $\phi \in W^{1,q+1}_0(U^{\prime})$, where $D_iw=\frac{\partial{w}}{\partial y_i}$, and
\begin{align} \label{def_coeff}
& \hat{\lambda}(y)=\lambda(\Phi(y)), \ \hat{k}(y)=k(\Phi(y)), \nonumber\\
& \hat{\sigma}_{ij}=\displaystyle \sum_{r=1}^{d} \frac{1}{\varrho(\Phi(y))}\frac{\partial \Phi_i}{\partial x_r}(\Phi(y))\frac{\partial \Phi_j}{\partial x_r}(\Phi(y)),\nonumber\\
&\hat{\xi}_{ij}(y)=\displaystyle \sum_{r=1}^{d}b(\Phi(y))(1-\delta(\Phi(y)))\frac{\partial \Phi_i}{\partial x_r}(\Phi(y))\frac{\partial \Phi_j}{\partial x_r}(\Phi(y)),\\
&\hat{\eta}_{ij}(y)=\displaystyle \sum_{r=1}^{d}b(\Phi(y))\delta(\Phi(y))\frac{\partial \Phi_i}{\partial x_r}(\Phi(y))\frac{\partial \Phi_j}{\partial x_r}(\Phi(y)). \nonumber
\end{align}
Note that $D_r\hat{\sigma}_{ij}, D_r\hat{\xi}_{ij}, D_r\hat{\eta}_{ij} \in L^\infty(U^{\prime})$ for $r\in\{1,\ldots, d-1\}$ since $\Psi$ and $\Phi$ are $C^{1,1}$ mappings and $C^{1,1}=W^{2,\infty}$ (cf. Chapter 2, Section 2.6.4, \cite{DZ}). It can be shown (cf. Section 6.3.2, \cite{evans}) that
\begin{equation} \label{coercivity}
\begin{aligned}
&\displaystyle \sum_{i,j=1}^d \hat{\sigma}_{ij} \phi_{i} \phi_j \geq K_1|\phi|^2, \ \displaystyle \sum_{i,j=1}^d \hat{\xi}_{ij} \phi_{i} \phi_j \geq K_1|\phi|^2, \\
&\displaystyle \sum_{i,j=1}^d \hat{\eta}_{ij} \phi_{i} \phi_j \geq K_1|\phi|^2, \ \ \forall (y,\phi) \in U^{\prime} \times \mathbb{R}^d.
\end{aligned}
\end{equation}
\noindent Next, we choose a domain $W^{\prime}$ such that $G \subset \subset W^{\prime} \subset \subset U^{\prime}$ and select a cut-off function such that 
\begin{align*}
\begin{cases}
\zeta=1 \ \text{on} \ G, \ \ \zeta=0 \ \text{on} \ \mathbb{R}^d \setminus W^{\prime},\\
0 \leq \zeta \leq 1.
\end{cases}
\end{align*}
Let $|l|>0$ be small and choose $r \in \{1,\ldots, d-1\}$. Note that since we consider directions parallel to the interface now, we have $D_r^l \hat{\lambda}=0$, $D_r^l \hat{k}=0$ and that there exists a constant $K_2>0$ such that 
\begin{align} \label{est_coeff} 
|D_r^l \hat{\sigma}_{ij}(y)|<K_2,  \ |D_r^l \hat{\xi}_{ij}(y)|<K_2, \
|D_r^l \hat{\eta}_{ij}(y)|<K_2
\end{align} 
for a.e. $y \in W^{\prime}$, $1 \leq i,j \leq d$ and sufficiently small $|l|$.\\

\noindent \textit{Step 2: Existence of second order derivatives $ D_j D_i w \in H^1(0,T;L^2(G^+))$, $j \neq d$.} We then use $\phi=-D_r^{-l}(\zeta^2D_r^l\dot{w}(t))$ as a test function in \eqref{eq:boundary_reg}, which, after integration with respect to time, results in 
\begin{align*}
&\frac{1}{2}\Bigl[\int_{W^{\prime}}\frac{1}{\hat{\lambda^l}}(1-2\hat{k}^lw^l)(\zeta D_r^l \dot{w})^2\, dy\Bigr]_0^t+\frac{1}{2}\Bigl[\int_{W^{\prime}}\displaystyle \sum_{i,j=1}^d \hat{\sigma}_{ij}^l \zeta^2 D_r^l D_i w D_r^l D_j w\, dy\Bigr]_0^t\\
&+\int_0^t \int_{W^{\prime}}\displaystyle \sum_{i,j=1}^d \hat{\xi}_{ij}^l \zeta^2 D_r^l D_i \dot{w} D_r^l D_j \dot{w}\, dy \, ds \\
=& \,\int_0^t \int_{W^{\prime}}\frac{2\hat{k}^l}{\hat{\lambda}^l} D_r^l w \ddot{w} \zeta^2 D_r^l \dot{w} \, dy \, ds+\int_0^t \int_{W^{\prime}}\frac{\hat{k}^l}{\hat{\lambda}^l}(\dot{w}^l+2\dot{w})(\zeta D_r^l \dot{w})^2 \, dy \, ds \\
&-\int_0^t \int_{W^{\prime}}\displaystyle \sum_{i,j=1}^d D_r^l (\hat{\sigma}_{ij})D_i w \,( \zeta^2D_r^lD_j\dot{w}+2\zeta D_j \zeta D_r^l \dot{w})\, dy \, ds\\
&-2\int_0^t \int_{W^{\prime}}\displaystyle \sum_{i,j=1}^d\hat{\sigma}^l_{ij} D_r^l D_i w \, \zeta D_j\zeta D_r^l \dot{w}\, dy \, ds
\\
&-\int_0^t \int_{W^{\prime}}\displaystyle \sum_{i,j=1}^d D_r^l (\hat{\xi}_{ij})D_i \dot{w} \,( \zeta^2D_r^lD_j\dot{w}+2\zeta D_j \zeta D_r^l \dot{w})\, dy \, ds\\
&- 2\int_0^t \int_{W^{\prime}}\displaystyle \sum_{i,j=1}^d \hat{\xi}_{ij} D_i \dot{w} \, \zeta D_j \zeta D_r^l \dot{w}\, dy \, ds\\
&-\int_0^t \int_{W^{\prime}}\displaystyle \sum_{i,j=1}^d D_r^l(|J_\Phi^T\nabla \dot{w}|^{q-1}\hat{\eta}_{ij}D_i\dot{w}) (\zeta^2 D_r^l D_j \dot{w}+2\zeta D_j \zeta D_r^l \dot{w})\, dy \, ds.
\end{align*}
We can estimate the last term on the right hand side as follows
\begin{align*}
&-\int_0^t \int_{W^{\prime}}\displaystyle \sum_{i,j=1}^d D_r^l(|J_\Phi^T\nabla \dot{w}|^{q-1}\hat{\eta}_{ij}D_i\dot{w}) (\zeta^2 D_r^l D_j \dot{w}+2\zeta (D_j \zeta) D_r^l \dot{w})\, dy \, ds \\
=& \, -\int_0^t \int_{W^{\prime}}\displaystyle \sum_{i,j=1}^d |(J_\Phi^T\nabla \dot{w})^l|^{q-1}\hat{\eta}^l_{ij} D_r^lD_i\dot{w} (\zeta^2 D_r^l D_j \dot{w}+2\zeta (D_j \zeta) D_r^l \dot{w})\, dy \, ds \\
&-\int_0^t \int_{W^{\prime}}\displaystyle \sum_{i,j=1}^dD_r^l(|J_\Phi^T\nabla \dot{w}|^{q-1}\hat{\eta}_{ij}) D_i\dot{w} (\zeta^2 D_r^l D_j \dot{w}+2\zeta( D_j \zeta) D_r^l \dot{w})\, dy \, ds \\
\leq& \,  C\|\nabla \dot{w}\|^{q-1}_{L^\infty(0,T;L^\infty(W^\prime))}(\displaystyle \sum_{i=1}^d  \|\zeta D_r^lD_i\dot{w}\|^2_{L^2(0,T;L^2(W^\prime))}+\|\nabla \zeta\|_{L^\infty(0,T;L^\infty(W^\prime))}\|\nabla \dot{w}\|^2_{L^2(0,T;L^2(W^\prime))} )
\end{align*}
with $C$ independent of $\nabla\zeta$.
\noindent 
Due to \eqref{est_coeff} the rest of the terms on the right hand side can be estimated analogously to the estimates in the proof of Theorem \ref{thm:interior_reg}, which for sufficiently small $\|\nabla \dot{w}\|_{L^\infty(0,T;L^\infty(W^\prime))}$ leads to
\begin{align*}
&\|\zeta D_r^l \dot{w}\|^2_{L^\infty(0,T;L^2(W^\prime))}+\displaystyle \sum_{i=1}^d \|\zeta D_r^l D_i w \|^2_{L^\infty(0,T;L^2(W^\prime))}
+\displaystyle \sum_{i=1}^d  \|\zeta D_r^l D_i \dot{w}\|^2_{L^2(0,T;L^2(W^\prime))} \\
\leq& \,C((1+\|\dot{w}\|_{L^\infty(0,T;L^\infty(W^\prime))})\|\nabla w\|^2_{L^\infty(0,T;L^2(W^\prime))} \\
&+(1+\|\nabla \dot{w}\|^{q-1}_{L^\infty(0,T;L^\infty(W^\prime))})\|\nabla \dot{w}\|^2_{L^2(0,T;L^2(W^\prime))}
+|\nabla \dot{w}(0)|^2_{L^2(W^\prime)}+|w(0)|^2_{H^2(W^\prime)}).
\end{align*}
Recalling the definition of $\zeta$ and employing Lemma \ref{reg_evans} yields $D_j D_i w \in H^1(0,T;L^2(G^+))$ for $1\leq i \leq d$, $1\leq j \leq d-1$. \\

\noindent \textit{Step 3: Existence of second order derivative $ D_d D_d w \in H^1(0,T;L^2(G^+))$.} It remains to show that $D_{dd}w:=D_d D_d w \in H^1(0,T;L^2(G^+))$. From \eqref{eq:boundary_reg}, after integration by parts, we obtain 
\begin{align} \label{boundary_eq}
 &\int_{G^+} \{\hat{\sigma}_{dd}D_dw(t)+\hat{\xi}_{dd}D_d\dot{w}(t)+|J_\Phi^T\nabla \dot{w}(t)|^{q-1}\hat{\eta}_{dd}D_d \dot{w}(t)\} D_d \phi \, dy \nonumber\\
 =&\,\int_{G^+}\Bigl\{-\frac{1}{\hat{\lambda}}(1-2\hat{k}w(t))\ddot{w}(t)+\frac{2\hat{k}}{\hat{\lambda}}(\dot{w}(t))^2  \\
 &+\displaystyle \sum_{j=1}^{d-1}\displaystyle \sum_{i=1}^d \Bigl(D_j(\hat{\sigma}_{ij}D_iw(t))+D_j(\hat{\xi}_{ij}D_i \dot{w}(t))
 +D_j(|J_\Phi^T\nabla \dot{w}(t)|^{q-1}\hat{\eta}_{ij}D_i \dot{w}(t))\Bigr)\Bigr\}\phi \, dy,\\
&=: \,\int_{G^+}\hat{f}(w)(t)\phi \, dy,
\nonumber
\end{align}
for $\phi \in C_0^{\infty}(G^+)$, a.e. in $[0,T]$. Since the right hand side of the equation is well-defined, we conclude that for a.e. $t \in [0,T]$ the weak derivative of $\hat{\sigma}_{dd}D_dw(t)+\hat{\xi}_{dd}D_d\dot{w}(t)+|J_\Phi^T\nabla \dot{w}|^{q-1}\hat{\eta}_{dd}D_d \dot{w}(t)$ with respect to $y_d$ exists on $G^+$. Furthermore, for a.e. $t \in [0,T]$ the weak derivative satisfies
\begin{equation} \label{weak_derivative}
\begin{aligned}
&-D_d(\hat{\sigma}_{dd}D_dw(t)+\hat{\xi}_{dd}D_d\dot{w}(t)+|J_\Phi^T\nabla \dot{w}|^{q-1}\hat{\eta}_{dd}D_d \dot{w}(t))
=\hat{f}(w)(t)
 \end{aligned}
\end{equation}
on $G^+$. From what we have shown, it follows that $\hat{f}(w) \in L^2(0,T;L^2(G^+))$. We set 
$$z(t):=\hat{\sigma}_{dd}D_dw(t)+\hat{\xi}_{dd}D_d\dot{w}(t)+|J_\Phi^T\nabla \dot{w}|^{q-1}\hat{\eta}_{dd}D_d \dot{w}(t),$$ and 
$$\tilde{\xi}_{dd}(t,D_d\dot{w}(t)):= \hat{\xi}_{dd}+|J_\Phi^T\nabla \dot{w}(t)|^{q-1}\hat{\eta}_{dd},$$ (suppressing in the notation dependence on $D_1\dot{w}(t),\ldots,D_{d-1}\dot{w}(t)$ which we already know to be smooth anyway) so that relation \eqref{weak_derivative} reads
$$
-D_d z(t)= \hat{f}(w)(t),$$
where
\begin{equation}\label{ODE}
z(t)=\hat{\sigma}_{dd}D_dw(t)+\tilde{\xi}_{dd}(t,D_d\dot{w}(t))D_d\dot{w}(t).
\end{equation}
Since $\hat{f}(w)\in L^2(0,T;L^2(G^+))$, and using the fact that $D_j D_i w \in H^1(0,T;L^2(G^+))$ for $1\leq i \leq d$, $1\leq j \leq d-1$, we have $z\in L^2(0,T;H^1(G^+))$. On the other hand, due to assuming that $u \in W^{1,\infty}(0,T;W^{1,\infty}(\Omega))$, we know also that $z \in L^\infty(0,T;L^\infty(G^+))$. Therefore we conclude that $$z \in L^2(0,T;H^1(G^+)) \cap  L^\infty(0,T;L^\infty(G^+)).$$
\noindent Since \eqref{ODE} represents an ODE (pointwise a.e. in space) for $D_d w(t)$, it can be resolved as follows:
\begin{equation} \label{Dddow}
\begin{aligned}
D_dw(t)=& \, \exp \Bigl(-\int_0^t\frac{\hat{\sigma}_{dd}}{\tilde{\xi}_{dd}(\tau,D_d\dot{w}(\tau))}\, d\tau\Bigr)
\Bigl(\int_0^t\frac{z(\tau)}{\tilde{\xi}_{dd}(\tau,D_d\dot{w}(\tau))}
\exp\Bigl(\int_0^\tau\frac{\hat{\sigma}_{dd}}{\tilde{\xi}_{dd}(\rho,D_d\dot{w}(\rho))}\, d\rho \Bigr)\, d\tau \\
&+ D_dw(0)\Bigr),
\end{aligned}
\end{equation}
and then also $D_d\dot{w}(t)$ can be expressed in terms of $z, \hat{\sigma}_{dd}, \tilde{\xi}_{dd}$:
\begin{equation}\label{Dddotw}
\begin{aligned} 
D_d\dot{w}(t)=& \,\frac{z(t)}{\tilde{\xi}_{dd}(t,D_d\dot{w}(t))}
-\frac{\hat{\sigma}_{dd}}{\tilde{\xi}_{dd}(t,D_d\dot{w}(t))}
\exp\left(-\int_0^t\frac{\hat{\sigma}_{dd}}{\tilde{\xi}_{dd}(\tau,D_d\dot{w}(\tau))}\, d\tau\right)\\
&\left(\int_0^t\frac{z(\tau)}{\tilde{\xi}_{dd}(\tau,D_d\dot{w}(\tau))}
\exp\left(\int_0^\tau\frac{\hat{\sigma}_{dd}}{\tilde{\xi}_{dd}(\rho,D_d\dot{w}(\rho))}\, d\rho\right)\, d\tau
+ D_dw(0)\right)
\end{aligned}
\end{equation}
Since the right hand side depends on $D_d\dot{w}(t)$, this is rather a fixed point equation than an explicit expression for $D_d\dot{w}(t)$.
We therefore consider the fixed point operator $\mathcal{T}:M\to M$, defined by the right hand side of \eqref{Dddotw}, i.e., $$\mathcal{T}=\mathcal{S}\circ\mathcal{R},$$ where $\mathcal{R}:M\to \hat{M}$, $\hat{M}=\{\phi \in M: \phi \geq K_1\}$, is the superposition operator associated with $\tilde{\xi}_{dd}$, i.e., $\mathcal{R}(v)(t)=\tilde{\xi}_{dd}(t,v(t))$, and $\mathcal{S}:\hat{M} \to M$
\begin{align*} 
\mathcal{S}(r)(t)=& \,\frac{z(t)}{r(t))}
-\frac{\hat{\sigma}_{dd}}{r(t)}
\exp\left(-\int_0^t\frac{\hat{\sigma}_{dd}}{r(\tau)}\, d\tau\right)
\left(\int_0^t\frac{z(\tau)}{r(\tau)}
\exp\left(\int_0^\tau\frac{\hat{\sigma}_{dd}}{r(\rho)}\, d\rho\right)\, d\tau
+ D_dw(0)\right).
\end{align*}
Note that both $\hat{\sigma}_{dd}$ and $\tilde{\xi}_{dd}$ are bounded from below by $K_1$ due to \eqref{coercivity}. We can then conclude that $\mathcal{T}$ is a self-mapping on 
\[
M=M_0=L^\infty(0,T;L^\infty(G^+)),
\]
and on 
\[
M=M_1=L^2(0,T;H^1(G^+))\cap L^\infty(0,T;L^\infty(G^+))\,,
\]
since $z \in L^2(0,T;H^1(G^+)) \cap L^\infty(0,T;L^\infty(G^+))$. Moreover, $$\mathcal{R}'(D_d\dot{w})
=\frac{\partial}{\partial v}
\Bigl(\hat{\xi}_{dd}+|J_\Phi^T(D_1\dot{w},\ldots,D_{d-1}\dot{w},v)|^{q-1}\hat{\eta}_{dd}\Bigr)\vert_{v=D_d\dot{w}}$$ is small if $\|\nabla \dot{u}_i\|_{L^{\infty}(0,T;L^\infty(\Omega_i))}$ and therefore $\|\nabla \dot{w}\|_{L^{\infty}(0,T;L^\infty(G^+))}$ is small. This implies that $\mathcal{T}$ is a contraction on 
\[
\tilde{M}_0=\{v\in M_0 \, : \ \|v-D_d\dot{w}\|_{L^\infty(0,T;L^\infty(G^+))}\leq\gamma\}
\] 
and on 
\[
\tilde{M}_1=\{v\in M_1 \, : \ \|v-D_d\dot{w}\|_{L^\infty(0,T;L^\infty(G^+))}\leq\gamma\}\,.
\] 
for $\gamma$ sufficiently small. Thus the fixed point equation $v=\mathcal{T}(v)$ has a unique solution $v_0$ in $\tilde{M}_0$ and it also has a uniqe solution $v_1$ in $\tilde{M}_1$, and both have to coincide $v_1=v_0$ by uniqueness on $\tilde{M}_0\supseteq \tilde{M}_1$. On the other hand, obviously $D_d\dot{w}$ lies in $\tilde{M}_0$, solves this fixed point equation and thus has to coincide with $v_0$, hence also with $v_1$. This proves that $D_d\dot{w}\in M_1\subseteq L^2(0,T;H^1(G^+))$. \\
\indent By transforming $w$ back to $u$, we can conclude that $u \in H^1(0,T;H^2(B^+))$. The assertion then follows from the fact that the boundary is compact and can be covered by a finite set of balls $\{B_{r_i/2}(x_i)\}_{i=1}^N$.
\end{proof}
\noindent Higher boundary regularity was obtained under the assumption that $u$ belongs to $W^{1,\infty}(0,T;W^{1,\infty}(\Omega))$; this was necessitated  by the presence of the $q$-Laplace damping term in the equation. The assumption is equivalent to assuming Lipschitz continuity of $u$ in time and space, i.e.  $u \in C^{0,1}(0,T;C^{0,1}(\overline{\Omega}))$ (see Theorem 4, Chapter 5, \cite{evans}). 
\subsection{Neumann problem for the coupled system}
It remains to show $H^2$-regularity up to the boundary for the Neumann problem \eqref{coupled_Neumann}. 
\begin{theorem} Let the assumptions of Corollary \ref{interior_reg_Neumann} hold, let $\partial \Omega$ and $\Gamma=\partial \Omega_+$ be $C^{1,1}$ regular and let $u$ be the weak solution of \eqref{coupled_Neumann}. Furthermore, assume that $g \in L^2(0,T;H^{1/2}(\partial \Omega))$. If $u \in W^{1,\infty}(0,T;W^{1,\infty}(\Omega))$ and $\|\nabla \dot{u}\|_{L^{\infty}(\Omega)(0,T;L^\infty(\Omega))}$ is sufficiently small, then $u_i \in H^1(0,T;H^2(\Omega_i))$, $i \in \{+,-\}$. 
\end{theorem}
\begin{proof}
We will show that $u_- \in H^1(0,T;H^2(\Omega_-))$, since the regularity on $\Omega_+$ follows as in the proof of Theorem  \ref{thm:boundary_reg_coupled_Dirichlet}. We begin again as before, by straightening the boundary around $x_0 \in \partial \Omega_- \setminus \partial \Omega_+$. There exists a ball $B=B_r(x_0)$ for some $r>0$ and a $C^{1,1}$-diffeomorphism $\Psi:B \rightarrow \Psi(B) \subset \mathbb{R}^d$ such that $\text{det}|\nabla \Psi|=1$, $U^{\prime}=\Psi(B \cap \Omega_-) \subset \mathbb{R}^d_+$ is an open set, and $\Psi(B\cap \partial \Omega_-) \subset \partial \mathbb{R}^d_+$. We change the variables and write
\begin{align*}
&y=\Psi(x), \ x \in B, \\
&x=\Phi(y), \ y \in U^{\prime}.
\end{align*} 
We denote
$
 G:=\Psi(B_{\frac{r}{2}}\cap \Omega_-) \subset \subset U^{\prime}$. 
We define
\begin{align*}
w(y,t):= u(\Phi(y),t), \quad (y,t) \in U^{\prime} \times [0,T],
\end{align*}
and transform the orginial equation from $(B \cap \Omega) \times [0,T]$ to $U^\prime \times [0,T]$:
\begin{equation} \label{eq:boundary_reg_Neumann}
\begin{aligned}
& \int_{U^{\prime}}\Bigl\{\frac{1}{\hat{\lambda}}(1-2\hat{k}w(t))\ddot{w}(t)\phi+\displaystyle \sum_{i,j=1}^d\Bigl(\hat{\sigma}_{ij}D_iw(t)D_j\phi+\hat{\xi}_{ij}D_i \dot{w}_i(t) D_j \phi \\
&+|J_\Phi^T\nabla \dot{w}(t)|^{q-1}\hat{\eta}_{ij}D_i\dot{w}(t)D_j  \phi\Bigr)-\frac{2\hat{k}}{\hat{\lambda}}(\dot{w}(t))^2\phi\Bigr\} \, dy=\int_{\partial U^\prime} g \phi  |J^{-T}_{\Phi}n|_{\mathbb{R}^d}\, dx,
\end{aligned}
\end{equation}
for a.e. $t \in [0,T]$, and all $\phi \in W^{1,q+1}(U^{\prime})$, and $\hat{\lambda}$, $\hat{k}$, $\hat{\sigma}_{ij}$, $\hat{\xi}_{ij}$, and $\hat{\eta}_{ij}$  are defined as in \eqref{def_coeff}. We then again use $\phi=-D_r^{-l}(\zeta^2D_r^l\dot{w}(t))$, $r \in \{1,\ldots, d-1\}$ as a test function and proceed with the estimates like in the proof of Theorem \ref{thm:boundary_reg_coupled_Dirichlet}. The only difference here is the need to estimate the boundary integral over $\partial U^\prime$ appearing in the weak form, therefore we focus our attention solely on estimating this term:
\begin{align*}
&-\int_0^t\int_{\partial U^\prime} g D_r^{-l}(\zeta^2D_r^l\dot{w})  |J^{-T}_{\Phi}n|\, dx \, ds \\
=& \, \int_0^t\int_{B\cap\partial \Omega_-} g D_r^{-l}(\zeta^2D_r^l\dot{u})  \, dx \, ds \\
\leq&  \, C \|g\|_{L^2(0,T;H^{1/2}(\partial \Omega_-))} \|\zeta^2D_r^l\dot{u}\|_{L^2(0,T;H^{1/2}(\partial \Omega_-))} \\
\leq& \, C \|g\|_{L^2(0,T;H^{1/2}(\partial \Omega_-))} \|\zeta D_r^l\dot{u}\|_{L^2(0,T;H^{1}(\Omega_-))}\\
\leq& \, \varepsilon \|\zeta D_r^l\dot{u}\|^2_{L^2(0,T;H^{1}(\Omega))}+\frac{1}{4\varepsilon}C^2 \|g\|^2_{L^2(0,T;H^{1/2}(\partial \Omega_-))}.
\end{align*}
Showing that $D_{dd}w \in H^1(0,T;L^2(G))$ follows as in the proof of Theorem \ref{thm:boundary_reg_coupled_Dirichlet}, since we use $\phi \in C^{\infty}_0(G)$ in \eqref{eq:boundary_reg_Neumann}.
\end{proof}
\section*{Acknowledgments} The financial support by the FWF (Austrian Science Fund) under grant P24970 is gratefully acknowledged as well as the support of the Karl Popper Kolleg "Modeling-Simulation-Optimization", which is funded by the Alpen-Adria-Universit\" at Klagenfurt and by the Carinthian Economic Promotion Fund (KWF).


\section*{References}
\begin{thebibliography}{00}
\bibitem{AEvans}
N. D. Alikakos and L. C. Evans, Continuity of the gradient of the solutions of certain dgenerate parabolic equations, J. Math. pures et appl., 62, 253--268, (1983). 
\bibitem{Ammari}
H. Ammari, D. Chen, J. Zou, Well-posedness of a Pulsed Electric Field Model in Biological Media and its Finite Element Approximation, arXiv:1502.06803.
\bibitem{BGT97}
A.~Bamberger, R.~Glowinski and Q.H.~Tran,
A domain decomposition method for the acoustic wave equation with discontinuous coefficients and grid change.
SIAM J. Numer. Anal. 34, 603--639, (1997).
\bibitem{BKR13}
R.~Brunnhuber, B.~Kaltenbacher and P.~Radu,
Relaxation of regularity for the Westervelt equation by nonlinear damping with application in acoustic-acoustic and elastic-acoustic coupling,
Evolution Equations and Control Theory, 3(4):595 - 626, (2014).
\bibitem{DZ}
M. C. Delfour and J.P. Zolesio, Shapes and Geometries, 2nd edn, SIAM, (2001).
\bibitem{Dibendetto}
E. DiBenedetto, Degenerate parabolic equations, Universitext, New York, NY: Springer-Verlag, xv 387, (1993.)
\bibitem{Friedman}
E. DiBendetto and A. Friedman, Regularity of solutions of nonlinear degenerate parabolic systems, J. Reine Angew. Math., 349, 83--128, (1984).
 \bibitem{evans}
\newblock L. C. Evans,
Partial Differential Equations,
2nd edn, American Mathematical Society, Providence, (1998).
\bibitem{Gao}
H. Gao and T. F. Ma, Global solutions for a nonlinear wave equation with the $p-$Laplacian operator, Electronic J. Qualitative Theory Differ. Equ., 11,
1–-13, (1999).
\bibitem{HamiltonBlackstock}
 M.~F. Hamilton and D.~T. Blackstock, 
Nonlinear Acoustics, 
 Academic Press, New York, (1997).
\bibitem{KL08}
B.\ Kaltenbacher and I.\ Lasiecka,
Global existence and exponential decay rates for the Westervelt equation,
Discrete and Continuous Dynamical Systems Series S, 2, 503--525, (2009).
\bibitem{KLV10}
B.\ Kaltenbacher, I.\ Lasiecka and S.\ Veljovi\'c, 
Well-posedness and exponential decay for the Westervelt equation with inhomogeneous Dirichlet boundary data,
J. Escher et al (Eds): Progress in Nonlinear Differential Equations and Their Applications, 60, 357--387, (2011).
\bibitem{manfred}
M.~Kaltenbacher,
Numerical Simulations of Mechatronic Sensors and Actuators, Springer, Berlin, (2004).
\bibitem{Lieberman}
G. M. Lieberman, Boundary regularity for solutions of degenerate elliptic equations, Nonlinear Anal., 12, 1203--1219, (1988).
\bibitem{lindqvist}
P. Lindqvist,
 Notes on the p-Laplace equation,
Lecture notes, University of Jyv\"{a}skyl\"{a}, (2006).
\bibitem{LiuYan}
W. Liu and N. Yan, Quasi-norm local error estimators for $p$-Laplacian, SIAM Journal on Numerical Analysis, 39, 100--127, (2002).
\bibitem{VN}
V. Nikoli\' c, Existence results for the Westervelt equation with nonlinear damping and Neumann as well as absorbing boundary conditions, J. Math. Anal. Appl., doi:10.1016/j.jmaa.2015.02.076, (2015).
\bibitem{NK15}
V. Nikoli\' c and B. Kaltenbacher, Sensitivity analysis for shape optimization of a focusing acoustic lens in lithotripsy, in preparation.
\bibitem{U}
N. N. Ural'tseva, Degenerate quasilinear elliptic systems, Zap.  Na.  Sem.  Leningrad.  Otdel.  Mat.  Inst.  Steklov.
(LOMI), 7, 184–-222, (1968).
\bibitem{Westervelt}
P.J. Westervelt,
Parametric acoustic array, 
The Journal of the Acoustic Society of America, 35: 535--537, (1963).
\bibitem{W}
Z. Wilstein, Global well-posedness for a nonlinear wave equation with $p$-Laplacian damping, PhD thesis, University of Nebraska-Lincoln, (2011).
\end{thebibliography}
\end{document}